\documentclass[times, 12pt, reqno]{amsart}

\usepackage{amsmath,amsfonts,amssymb}
\usepackage[colorlinks=true,breaklinks=true]{hyperref}
\usepackage{amsbsy}
\usepackage{eufrak}
\usepackage{xcolor}
\usepackage{amsopn}
\usepackage{amsthm}
\usepackage[english]{babel}
\usepackage[applemac]{inputenc}
\usepackage[T1]{fontenc}
\usepackage{graphicx}
\usepackage{times}
\usepackage{enumerate}

\DeclareMathOperator{\er}{\mathbb{E}}
\DeclareMathOperator{\pr}{\mathbb{P}}
\DeclareMathOperator{\du}{{\rm d}\it u}
\DeclareMathOperator{\dv}{{\rm d}\it v}
\DeclareMathOperator{\dw}{{\rm d}\it w}
\DeclareMathOperator{\dx}{{\rm d}\it x}
\newcommand{\simdis}{\stackrel{d}{=}}
\newcommand{\II}{{\rm  1~\hspace{-1.4ex}l}}

\newcommand\NN{\mathbb{N}_0}
\newcommand\No{\mathbb{N}}
\newcommand\RR{\mathbb{R}}
\newcommand\RP{[0,\infty)}
\newcommand\CM{\mathcal{C\!M}}
\newcommand\CA{\mathcal{C\!A}}

\newcommand\BF{\mathcal{BF}}
\newcommand\CF{\mathcal{CF}}
\newtheorem{theorem}{Theorem}
\newtheorem{definition}{Definition}
\newtheorem{remark}{Remark}
\newtheorem{corollary}{Corollary}
\newtheorem{lemma}{Lemma}
\newtheorem{example}{Example}
\newtheorem{proposition}{Proposition}

\newcommand\rsetminus{\mathbin{\mathpalette\rsetminusaux\relax}}
\newcommand\rsetminusaux[2]{\mspace{-4mu}
  \raisebox{\rsmraise{#1}\depth}{\rotatebox[origin=c]{-20}{$#1\smallsetminus$}}
 \mspace{-4mu}
}
\newcommand\rsmraise[1]{%
  \ifx#1\displaystyle .8\else
    \ifx#1\textstyle .8\else
      \ifx#1\scriptstyle .6\else
        .45%
      \fi
    \fi
  \fi}
\newcommand\msetminus{\rsetminus\!}
\title[$\CM$ and $\BF$ properties of functions are characterized by their restriction on $\NN$]{Complete  monotonicity and Bernstein properties of functions are characterized by their restriction on $\NN$}

\author[Rafik Aguech  and Wissem Jedidi]{Rafik Aguech$^{\diamond,\star}$ and Wissem Jedidi$^{\diamond,\star \star}$}
\address{\rm $^{\diamond}$ Department of Statistics \& OR, King Saud University, P.O. Box 2455, Riyadh 11451, Saudi Arabia.}
\address{\rm $^{\star}$ Universit\'e de Monastir, Facult\'e des Sciences de Monastir, D\'epartement de math\'ematiques, 5019 Monastir, Tunisie. {\em Email}: {\tt rafik.aguech@ipeit.rnu.tn}}
\address{\rm $^{\star \star}$ Universit\'e de Tunis El Manar, Facult\'e des Sciences de Tunis, LR11ES11 Laboratoire d'Analyse Math\'ematiques et Applications, 2092, Tunis, Tunisie.  {\em Email} : {\tt wissem\_jedidi@yahoo.fr}}

\date{\today}
\begin{document}
\maketitle
\begin{abstract} We give several new characterizations  of completely monotone functions ($\CM$) and Bernstein functions ($\BF$) via two approaches:
the first one is driven algebraically via  elementary preserving mappings and the second one is developed in terms of the behavior of their  restriction on $\NN$. We give a complete answer to the following question: {\it Can we affirm that a function $f$  is completely monotone (resp. a Bernstein function) if we know that the sequence ${\left(f(k)\right)}_{k}$ is completely monotone (resp. alternating)?} This approach constitutes a kind of converse of Hausdorff's moment characterization theorem  in the context of completely monotone sequences.
\end{abstract}
{\footnotesize {\bf Keywords}: Completely monotone functions, completely monotone sequences, Bernstein functions, completely alternating functions, completely alternating  sequences, Hausdorff moment problem, Hausdorff moment sequences, self-decomposability.}
\section{Introduction}
Traditionally, completely monotone functions ($\CM$) are recognized as Laplace transforms of positive measures and Bernstein functions  ($\BF$) are their positive antiderivatives. The literature devoted to these two classes of functions is impressive since they have remarkable applications in various branches, for instance, they play a role in potential theory, probability theory, physics, numerical and asymptotic analysis, and combinatorics.  A detailed collection of the most important properties of completely monotone functions can be found in the monograph of Widder \cite{widder} and for Bernstein functions, the reader is referred to the elegant manuscript of Schilling,  Song and Vondra\v{c}ek \cite{SSV}. Hausdorff's moment characterization theorem \cite{haus} is explained in details, and also in the context of measures on commutative semigroup in the Book of Berg, Christensen and Ressel \cite{berg1}.  The references \cite{berg1} and \cite{SSV} were a major support in the elaboration of this paper and constitute for us a real source of inspiration.

Theorem \ref{bergh} below, is borrowed from \cite{berg1} and gives the complete characterization of completely monotone (respectively alternating) sequences: a sequence ${(a_k)}_{k}$ is interpolated by a function $f$ in $\CM$ (respectively $\BF$) if and only if ${(a_k)}_{k}$  completely monotone (respectively alternating) sequence and minimal (see Definition \ref{WA} for minimality). Completely monotone sequences are also known as  the Hausdorff moment sequences. In this spirit, a natural question prevailed, {\it what about the converse?} i.e:
\begin{quote}{\it
Can we affirm that a function $f$ belongs to   $\CM$ (respectively $\BF$) if we know that the sequence ${\left(f(k)\right)}_{k}$ is completely monotone (respectively alternating)? In other terms, could a  completely monotone (respectively alternating)  and minimal sequence ${(a_k)}_{k}$ be interpolated by a  regular enough function $f$, which is not in $\CM$ (respectively $\BF$)?
}\end{quote}
We prove that under natural regularity assumptions on $f$, the answer is affirmative for the first question (and then infirmative for the second) and this constitutes a kind of converse of Hausdorff's moment characterization theorem \cite{haus}.  Mai, Schenk and Scherer \cite{mai} adapted  a Widder's  result \cite{widder} and used a specific technique from Copula theory in  order to state, in  their Lemma 3.1 and Theorem 1.1,  that:

\noindent (i) a continuous  function $f$ with $f(0)=1$ belongs to $\CM$  if and only if the sequence ${\left(f(xk)\right)}_{k}$ is completely monotone for every $x \in\mathbb{Q}\cap [0,\infty)$;

\noindent (ii) a continuous  function $f$ with $f(0)=0$ belongs to $\BF$ and is self-decomposable if and only if the sequence
${\left(f(xk)- f(yk)\right)}_{k}$ is completely alternating for every $x>y>0$. (See Section \ref{self} below for the definition of self-decomposable Bernstein functions).

The idea of this paper was born when we wanted to  remove  the dependence on $x$ in characterizations (i) and (ii) and  to study general non bounded completely monotone functions and general Bernstein functions. Our answer to the question is given  in Theorems \ref{theorem4} and \ref{theorem5} below that say:

\noindent (iii) a bounded function  $f$ belongs to $\CM$  if and only if it has an holomorphic extension on $Re(z)>0$ which remains bounded there and  the sequence $\big(f(xk)\big)_{k\geq 0}$ is completely monotone and minimal for some (and hence for all)  $x>0$. If  $f$ is unbounded, then a shifting condition  is necessary;

\noindent (iv) a bounded function $f$   is a Bernstein function if and only if  it has an holomorphic extension on $Re(z)>0$,  and  the sequence
$\big(f(xk)\big)_{k\geq 0}$ is completely alternating and minimal for some (and hence for all)  $x>0$.  If  $f$ is unbounded, then a boundedness condition on the increments is necessary.

For each of Theorems \ref{theorem4} and \ref{theorem5} we shall give two proofs based on two different approaches, the first one uses Blaschke's theorem on the zeros of a function on the open unit disk and the second one is based on a Greogory-Newton expansion of holomorphic functions (see Section \ref{prerequisite} below for the last two concepts). We emphasize that these two approaches require some boundedness (especially in the completely monotone case). In Corollary 4.2 of Gnedin and Pitman \cite{gned}  the necessity part of (iv) above is stated without the holomorphy and minimality condition, their formulation is equivalent to Theorem \ref{bergh} below. We discovered the idea of our second proof  (for the Bernstein property context) hidden in the remark right after their corollary. The authors surmise that the sufficiency part of (iv) could be proved by  Gregory-Newton expansion of Bernstein functions and we will show that their idea works.  Since we are studying general, non necessarily bounded functions in $\CM$ and in $\BF$, there was a price to pay in order to avoid these kind of restrictive conditions. For this purpose, we develop in Section \ref{toolfunction} and \ref{toolsequence} there several algebraic tools, based on the scale, shift and difference operators, giving new characterizations for the $\CM$ and $\BF$ classes. We did our best to remove redundant assumptions of regularity (such as continuity or differentiability or boundedness or global dependence on parameters) in the our sufficiency conditions. This kind of redundancy often appears, because the classes $\CM$ and $\BF$ are very rich in information. These tools, that we find intrinsically useful, can also be considered as a major contribution in this work. They were also crucial in the proofs of the results given in Section \ref{main}. Throughout this paper, we give different proofs, whenever it is possible, and when the approaches  were clearly distinct.   \\

The paper is organized as follows. Section \ref{basic} gives the basic setting and definitions. In Section \ref{toolfunction} and \ref{toolsequence}, we recall classical characterizations of complete monotonicity and alternation for functions and sequences, we develop several other characterizations  and we discuss the concept of minimal sequences.  Section \ref{prerequisite} is devoted to specific pre-requisite for the proofs of the main results. We recall there and adapt some results around functional iterative equations and asymptotic of differences of functions. We also adapt some results stemming from complex analysis and  from interpolation theory. Section \ref{proof} is devoted to the proofs and Section \ref{self}  gives an alternative characterization  for self-decomposable Bernstein functions to point point (ii) above, in the spirit of point (iv) above.
\section{Basic notations and definitions}\label{basic}

Throughout  this paper, $\NN$ denotes the set of non-negative integers and $\No=\NN\msetminus\{0\}$. A sequence ${(a_k)}_{k\in \NN}$ is seen as a function $a:\NN \to\mathbb{R}$ so that $a(k)=a_k$. The symbols $\wedge$ and $\vee$ denote respectively the $\min$ and the $\max$. All the considered functions are measurable, the measures   are positive, Radon with support contained in $[0,\infty)$.  For functions $f: D\subset \mathbb{C}\to \mathbb{C}$, the scaling, the shift and the difference operators acting on them are respectively denoted, whenever these are well defined, by
$$\begin{array}{ccll}
\sigma_c f(x) &:=& f(cx),  & \quad\sigma =\sigma_1={\rm \it Identity}, \\
\tau_c f(x) &:=& f(x+c),  & \quad\tau =\tau_1, \\
\Delta_c f(x) &:=& f(x+c)-f(x),  & \quad\Delta =\Delta_1,\\
\theta_c f(x) &:=& f(c)- f(0)+ f(x)- f(x+c),  & \quad\theta =\theta_1,
\end{array}$$
and  their iterates are given by
$\sigma_c^0f=\tau_c^0f=\Delta_c^0 f=\theta_c^0f=f$ and for every $n\in \No$,
$$\sigma_c^n=\tau_{c^n},\quad \tau_c^n=\tau_{cn},\quad\Delta_c^n f = \Delta_c(\Delta_c^{n-1} f),\quad \theta_c^n f= (-1)^n \big(\Delta_c^n f-\Delta_c^n
f(0)\big),$$ \noindent so that for every $n\in \NN,$
\begin{eqnarray}\label{delta}
\Delta_c^n f(x)&=&\sum_{i=0}^n \,\binom{n}{i} \,(-1)^{n-i} \, f(x+ic)\\
\theta_c^n f&=&\sum_{i=0}^n \,\binom{n}{i} \,(-1)^{i} \,\big(f(x+ic)-f(ic)\big).\nonumber
\end{eqnarray}
\begin{definition}[Berg \cite{berg1} p. 130] Let $D=(0,\infty)$ or $[0,\infty)$ or $\NN$. A function $f:D\to f(D)$ is called completely monotone on $D$, and we denote $\;f\in \CM(D),\;$  if $\;f(D)\subset [0,\infty)\;$ (respectively completely alternating if  $\;f(D)\subset \mathbb{R}$ and we  denote $\;f\in \CA(D)$\;), if for  all finite sets $\;\{c_1, \cdots, c_n\}\subset D\;$ and $\;x\in D\;$, we have
$$(-1)^n \Delta_{c_1}\cdots \Delta_{c_n}f(x)\geq 0 \quad  \mbox{\it (respectively}\;\; \leq 0).$$
\label{defi}\end{definition}
\begin{remark}  (i) Every function  $f$ in  $\CM(D)$  (respectively $\CA(D)$) is non-increasing (respectively non-decreasing). We will see later on that $f$ is necessarily decreasing (respectively increasing) when it is not a constant.

(ii) A non-negative function $f$ belongs to $\CM(D)$   if and only if $-\Delta_cf $ belongs to $\CM(D)$ for every $c \in D\msetminus  \{0\}$.

(iii) By \cite[Lemma 6.3 p. 131]{berg1}, a function $f$ belongs to $\CA(D)$ if and only if for every $c\in D\msetminus \{0\}$, the
function $\Delta_c f$ belongs to $\CM(D)$.

(iv) By  linearity of the difference operators, the classes $\CM(D)$ and $\CA(D)$ are convex cones.\label{r1}
\end{remark}
\section{Classical characterizations of completely monotone and alternating functions and additional characterizations via algebraic transformations}\label{toolfunction}
\subsection{Completely monotone functions} Characterization of completely monotone functions is an old story and is due to the seminal works of Bernstein, Bochner and Schoenberg.  A nice presentation could be found in the monograph of Schilling et al. \cite{SSV}:
\begin{theorem} \cite[Proposition 1.2 and Theorem 4.8]{SSV} The following three assertions are equivalent:
\begin{enumerate}[\;(a)]
\item $\Psi$   is completely monotone on $(0,\infty)$ (respectively on
$[0,\infty)$);

\item $\Psi$  is represented as the Laplace transform of a unique Radon (respectively finite) measure $\nu$ on $[0,\infty)$:
\begin{equation}\label{lap}
\Psi(\lambda)=\int_{[0,\infty)} e^{-\lambda x} \nu (\dx), \quad \lambda >0\;\; \mbox{(respectively    $\; \lambda \geq0$)};
\end{equation}

\item $\Psi$  is  infinitely differentiable  on $(0,\infty)$ (respectively continuous on $[0,\infty)$,  infinitely differentiable  on $(0,\infty)$) and satisfies  $(-1)^n \,\Psi^{(n)}\geq 0$ for every $n \in \NN$.
\end{enumerate}\label{comf}\end{theorem}
The measure $\nu$ in (\ref{lap}) will be referred in the sequel as the representative measure of $\Psi$.
\begin{remark}  (i) Every function $\Psi \in \CM(0,\infty)$ such that $\Psi(0+)$ exists, naturally extends to a continuous bounded function in
$\CM[0,\infty)$, this is the reason why we identify, throughout this paper,  such functions $\Psi$ with their extension on $[0,\infty)$.

(ii) By Corollary 1.6 p. 5 in \cite{SSV}, the closure of $\CM[0,\infty)$ (with respect to pointwise convergence) is $\CM[0,\infty)$. This insures that $\Psi \in \CM(0,\infty)$ if and only if $\tau_{c_n}\Psi \in \CM[0,\infty)$ for some positive sequence  $c_n$ tending to zero or equivalently $\tau_c f  \in \CM(0,\infty)$ for every $c>0$. It is also immediate that $\Psi \in \CM(0,\infty)$ if and only if  $\sigma_c f  \in \CM(0,\infty)$ for some (and hence for all) $c>0$.

(iii) It is not clear at all to see that   functions  in $\CM(0,\infty)$ are actually infinitely differentiable just using Definition \ref{defi}. The latter is clarified by point (b) of Theorem \ref{comf}.  Furthermore,  Dubourdieu \cite{dubour} pointed out that strict inequality prevails in point $(c)$ of for all non-constant completely monotone functions, for these and their derivatives are then strictly monotone.
\label{clos}\end{remark}

We start with a taste of what we can obtain as algebraic characterization. The following proposition  has to be compared with the Remark \ref{r1} (ii):
\begin{proposition} (a) A function $\Psi:(0,\infty) \to [0,\infty)$ belongs to $\CM(0,\infty)$  if and only if for some (and hence for all)  $c >0$ the function $\,-\Delta_c\Psi\,$  belongs to  $\CM(0,\infty)$ and the Laplace representative  measure in (\ref{lap}) of $\,-\Delta_c\Psi\,$ gives no mass to zero.

\noindent (b) In this case, the sequence of functions $(-\Delta_{nc})\Psi$ converges pointwise, locally uniformly, to a function in $\CM(0,\infty)$ that does not depend on $c$, more precisely $$\Psi(\lambda) =\lim_{x\to \infty}\Psi(x) + \lim_{n\to \infty}(-\Delta_{nc})\Psi(\lambda),\quad \lambda >0.$$
The same holds for the successive derivatives of $(-\Delta_{nc})\Psi$.
\label{proposition1}\end{proposition}
\subsection{Completely alternating functions and Bernstein functions} The well known class $\BF$ of Bernstein functions consists of those functions  $\Phi:(0,\infty) \to [0,\infty)$,  infinitely differentiable on $(0,\infty)$   and satisfy  $(-1)^{n-1} \Phi^{(n)}(\lambda)\geq 0$, for every $\lambda>0$ and $n\in \No$. In other terms, $\Phi$ is a Bernstein function if it is  non-negative, infinitely differentiable and $\Phi' \in \CM(0,\infty)$.  It is also known (see Theorem 3.2 p. 21\cite{SSV} for instance) that any function $\Phi \in \BF$ admits a continuous extension on $[0,\infty)$, still denoted
$\Phi$,  and represented by
\begin{equation}\label{berep}
\Phi(\lambda)=q+ d\lambda +\int_{(0,\infty)}(1-e^{-\lambda x})\mu(\dx),   \quad \lambda \geq 0\,,
\end{equation}
\noindent where $q,\,d\geq 0$ and the so-called  L\'evy measure $\mu$ satisfies the integrability condition
$$\int_{(0,\infty)} (1-e^{-x})\, \mu(\dx) <\infty\quad \mbox{which is equivalent to}\quad \int_{(0,\infty)}\left(1 \wedge x \right) \mu(\dx) <\infty.$$
An integration by parts gives
$$\int_{(0,\infty)} e^{-\lambda x} \mu\big((x,\infty)\big) \,\dx = \frac{\Phi(\lambda)-q}{\lambda }-d,   \quad \lambda >0,$$
so that $q=\Phi(0),\; d=\lim_{x\to \infty} \frac{\Phi(x)}{x}$ and  the relation between $\Phi$ and the triplet $(q,d,\mu)$ becomes one-to-one. \\

The following proposition unveils the link between completely alternating functions and Bernstein functions:
\begin{proposition} 1) The class of Bernstein functions coincides with the class of non-negative and completely alternating functions on  $[0,\infty)$.

2) The class of completely alternating functions on  $(0,\infty)$ is given by
$$\CA(0,\infty)=\{f:(0,\infty)\to \RR,\; \mbox{\it differentiable}, \; s.t.\; f'\in \CM(0,\infty)\}.$$
In particular, if $g\in \CM(0,\infty)$, then $-g\in \CA(0,\infty)$.
\label{proposition2}\end{proposition}

It is clear that the subclass $\BF_b$ of bounded Bernstein function is given by
\begin{eqnarray*}
\BF_b&=&\{\Phi \in \BF, \; s.t.\; \lim_{\lambda \to \infty}\Phi(\lambda)<\infty\}\\
&=&\{\Phi \in \BF, \; s.t.\; \Phi(\lambda)=q+\int_{(0,\infty)} (1-e^{-x\lambda})\, \mu(\dx), \;\;  \mbox{with}\; q\geq 0, \; \mu\big((0,\infty)\big) <\infty\}
\end{eqnarray*}
and  that
\begin{equation}\Phi\in \BF_b\quad\mbox{\it if and only if $\quad\Phi \geq 0\;$ and}\; \;\Phi(\infty)-\Phi \in \CM[0,\infty).
\label{bcm}\end{equation}
We denote
\begin{eqnarray*}
\BF_b^0&=&\{\Phi \in \BF_b,\; s.t.\; \Phi(0)=0\}\\
&=&\{\Phi \in \BF, \; s.t.\; \Phi(\lambda)=\int_{(0,\infty)} (1-e^{-x\lambda})\, \mu(\dx), \;\;   \mbox{with}\; \mu\big((0,\infty)\big)  <\infty\}.
\end{eqnarray*}

We also have the following equivalences
\begin{eqnarray}
 \Phi \in \BF  &\Longleftrightarrow&\Phi \geq 0 \;\mbox{and}\; \sigma_c\Phi \in \BF  \;\mbox{for some (and hence for all)}\;  c>0 \label{bc}\\
&\Longleftrightarrow&\lambda \mapsto \Phi(\lambda+c) - \Phi(c)\in \BF,\quad\mbox{for every}\;  c>0. \label{pc}
\end{eqnarray}
Equivalence (\ref{bc}) immediate and (\ref{pc}) is justified as follows: by differentiation   get  $\Phi'(.+c) \in \CM[0,\infty)$,  for all $c>0$ and  closure of the class $\CM(0,\infty)$ (Corollary 1.6 p.5 \cite{SSV}) insures that $\Phi' \in  \CM(0,\infty)$. A natural question is to ask whether (\ref{pc}) remains true if expressed with a single fixed $c>0$. The answer is negative because for every $\Phi_0 \in \BF $, the function $\Phi(\lambda)=\Phi_0(|\lambda -c|),\;
\lambda \geq 0,$ is not in $\BF$ despite that $\lambda \mapsto \Phi(\lambda+c) - \Phi(c) \in \BF$. A closed transformation is studied in Corollary 3.8 (vii) p. 28 in \cite{SSV} which says that $\Phi \in \BF$ yields $\theta_c\Phi \in \BF$ for every  $c>0$. We propose the following improvement:
\begin{proposition} (a) A function $\Phi: [0,\infty)\longrightarrow \RP$ belongs to $\BF$ if and only if for some (and hence for all) $c>0$,
$$\lambda \mapsto \theta_c\Phi(\lambda) = \Phi(c)-\Phi(0)+\Phi(\lambda)-\Phi(\lambda+c)\in \BF_b^0.$$
\noindent \noindent (b) In this case, the sequence of functions $\theta_{nc} \Phi$  converges pointwise, locally uniformly, to a function in $\BF$, null in zero, that does not depend on $c$. More precisely
$$\Phi(\lambda) = \Phi(0) + \lambda \lim_{x\to \infty} \frac{\Phi(x)}{x} + \lim_{n\to \infty} \theta_{nc} \Phi(\lambda),\quad \lambda \geq 0.$$
The same holds for the successive derivatives of $\theta_{nc} \Phi$.
\label{proposition3}\end{proposition}
\begin{remark} By (\ref{bcm}), point (a) is also equivalent to $\lambda \mapsto  \Phi(\lambda+c) -\Phi(\lambda)\in \CM[0,\infty)$, for some (and hence for all) $c>0$.
\end{remark}
\section{Classical characterization of completely monotone and alternating sequences and additional results}\label{toolsequence}

A characterization of completely monotone (respectively alternating) sequences, closely related to Hausdorff moment characterization theorem  \cite{haus},  could be found in the monograph of Berg et al. \cite{berg1}:
\begin{theorem}\cite[Propositions 6.11 and 6.12 p. 134]{berg1} Let $a={\left(a_k\right)}_{k\geq 0}$ a positive sequence. Then,  the following conditions are equivalent: \vspace{-0.2cm}
\begin{enumerate}[(a)]
\item the sequence $a$ is completely  monotone (respectively alternating);

\item for all $k \in \NN,\,  n\in \NN$ (respectively $n\geq 1$), we have
\begin{equation}(-1)^n \Delta^n a(k)   \geq 0 \quad\mbox{(respectively  $\leq 0$)};
\label{bergd} \end{equation}

\item there exists a positive Radon measure $\nu$ on $[0,1]$ (respectively $q\in \RR,\,d\geq 0$ and a positive Radon measure $\mu$ on $[0,1)$) such that we have the representation
\begin{eqnarray}
&&a_0=  \nu([0,1]) ,\quad   a_k=\int_{(0,1]} u^{k} \nu(\du),\;\; k\geq 1\label{bergu1}\\
&&\mbox{\it \Big(respectively} \;\;a_0=q,\quad   a_k=q +d\,k+\int_{[0,1)} (1-u^{k}) \, \mu(\du),  \;\;k\geq 1\Big). \label{bergu2}
\end{eqnarray}
\end{enumerate}\label{bergh}\end{theorem}
\subsection{Comments on $\CM(\NN)$ and $\CA(\NN)$}\label{comments}
$\,$

\underline{{\bf Comment 1}}: In the completely monotone case, the measure  $\nu$ in (\ref{bergu1}) is not only Radon but also finite because of the convention $a_0=\nu\left([0,1]\right)$. In the completely alternating case, we have that $a_0=q$ and the measure $\mu$ in (\ref{bergu2}) is only  Radon, satisfying the integrability  condition  $\int_{[0,1)} (1-u) \, \mu(\du)<\infty$. By the dominated convergence theorem, we retrieve $d=\lim_{k\to
\infty}\left(a_k/k\right).$ Furthermore, in both cases, $\nu$ (respectively $(q,d,\mu)$) uniquely determine the sequence ${\left(a_k\right)}_{k\geq 0}$, which is justified as follows:

1- In the completely monotone case: use Fubini argument, get that the exponential generating function of the sequence ${\left(a_k\right)}_{k\geq 0}$ is the Laplace transform of $\nu$,
$$ \sum_{k\geq 0} a_k\, \frac{(-t)^k}{k!}= \int_{[0,1]} e^{-tu}\, \nu(\du), \quad t\geq 0,$$
and finally conclude with the injectivity of the Laplace transform.

2- In the completely alternating case: making an integration by parts, write
$$a_k - q-dk - \mu \left(\{0\} \right) =k \int_0^1 u^{k-1}\, \mu\big((0,u]\big)  \du,\quad k\geq 1,$$
then by a Fubini argument, get that the exponential generating function of the sequence ${\left(a_k\right)}_{k\geq 0}$ leads to a Bernstein function build with the triplet $(q,\,d,\,\mu)$:
\begin{eqnarray}
h(t) :&=&\sum_{k\geq 0}  a_k \, \frac{t^{k}}{k!},\qquad t\geq 0\nonumber\\
&=& (q+d \,t)e^t+   \mu \left(\{0\} \right) (e^t-1)   +t \int_0^1 e^{tu}\, \mu\big((0,u]\big)  \du \nonumber\\
&=& (q+d \,t)e^t+ \mu \left(\{0\} \right) (e^t-1)   +  \int_{(0,1)} (e^{t} -e^{tv})\, \mu(\dv)  \nonumber\\
e^{-t}\,h(t) &=&q+  d \,t  +  \mu \left(\{0\} \right)(1 -e^{-t}) +   \int_{(0,1)} (1 -e^{-tw})\, \widehat{\mu} (\dw) \,\label{h}
\end{eqnarray}
where $\widehat{\mu}$ is the image of the measure $\mu$ obtained by the change of variable $w=1-v$, and finally conclude with the unicity through  the Bernstein representation in equality (\ref{h}).

\underline{{\bf Comment 2}}: Completely monotone sequences are always positive, whereas a completely alternating sequence is non-negative if and only if the corresponding $q$-value in (\ref{bergu2}) is non-negative
(see \cite{atha}).
\subsection{The classes $\CM^*(\NN)$ and $\CA^*(\NN)$ of minimal completely monotone and alternating sequences}\label{minima} A lot of care is required if one modifies some terms of  a completely monotone or alternating sequence. We clarify, with our own approach, the following fact we have found  in \cite{lorch1} and  \cite{lorch2}, and extend it to completely alternating sequences:  strict inequality prevails throughout (\ref{bergd}) for a completely monotone sequence unless $a_1=a_2 = \cdots= a_n = \cdots$, that is, unless all terms except possibly its first are identical. We can state that
\begin{quote}{\it A sequence  $a={\left(a_k\right)}_{k\geq 0}$ in $\CM(\NN)$ (respectively in $\CA(\NN)$) ceases to strictly alternate, in differences,  at a certain rank if and only if the sequence  $a$  is constant (respectively  if and only if if the sequence  $a$ is affine)}.
\end{quote}
Our argument uses the  explicit computation (\ref{delta}) of  the quantities $(-1)^n \Delta^n a(k), \, n\in \No,\,k\in \NN$, which does not seem  to be fully exploited in the literature we encountered:
$$(-1)^n \Delta^n a(k)=\left\{
\begin{array}{ll}
\;\; \,\mu(\{0\})\,\II_{k=0}+ \int_{(0,1)}u^k(1-u)^n \nu(\du)   & {\rm if }\; a \in \CM(\NN) \\
-\mu(\{0\})\,\II_{k=0}- \int_{(0,1)}u^k(1-u)^n \mu(\du)& {\rm if }\; a \in \CA(\NN).
\end{array} \right.$$
Let $\alpha=\nu$ or $\mu$. Based on the fact that $\int_{(0,1)}u^k(1-u)^n \alpha(\du) =0$,   for some $n\in \No$ and
$k\in \NN$, if and only if $\alpha\big((0,1)\big)=0$, then an elementary reasoning shows that
$$(-1)^n \Delta^n a(k)= 0 \;\mbox{for some} \,n\in \No,\,k\in \NN \Longleftrightarrow \left\{ \begin{array}{lll}
a_m =\mu(\{1\}),& \forall m\geq 1  & {\rm if }\; a \in \CM(\NN) \\
a_m =q+\mu(\{0\})+d\,m,&\forall m\geq 1 & {\rm if }\; a \in \CA(\NN).
\end{array} \right.$$

As an example, fix $\epsilon>0$ and consider the completely monotone (respectively alternating) sequence $b_0=\epsilon, \, b_k=0,\,k\geq 1$ (respectively $b_0=0, \, b_k=\epsilon,\,k\geq 1$). It satisfies:
$$(-1)^n \Delta^n b(k)= \epsilon  \,\II_{k=0}\; (\mbox{respectively}\;-\epsilon \,\II_{k=0}),\quad   n\in \No,\, k \in \NN.$$
By linearity of the operators $(-1)^n\Delta^n$, we obviously have
$$(-1)^n \Delta^n (a-b)(k)=\left\{
\begin{array}{ll}
\;\; \,(\mu(\{0\})-\epsilon)\,\II_{k=0}+ \int_{(0,1)}u^k(1-u)^n \nu(\du)   & {\rm if }\; a \in \CM(\NN) \\
\big(\epsilon-\mu(\{0\})\big)\,\II_{k=0}- \int_{(0,1)}u^k(1-u)^n \mu(\du)& {\rm if }\; a \in \CA(\NN).
\end{array} \right.$$
Since $\nu$ is finite (respectively $\mu$ integrates $1-u$), then the dominated convergence theorem ensures  that
$$\lim_{n\to \infty}\int_{(0,1)} (1-u)^n \nu(\du)=\lim_{n\to \infty}\int_{(0,1)} (1-u)^n \mu(\du)=0,$$
so that the quantities $(-1)^n \Delta^n (a-b)(0)$ takes the sign of $\mu(\{0\})-\epsilon$ when $n$ is big enough. The above discussion clarifies the concept of minimality initially introduced,  with a different approach,  in the monograph of Widder \cite{widder}:
\begin{definition}\cite[Widder, p. 163]{widder} and \cite[Athavale-Ranjekar]{atha}: Let $a={\left(a_k\right)}_{k\geq 0}$ a completely monotone
(respectively alternating) sequence. \vspace{-0.2cm}
\begin{enumerate}[(i)]
\item $a$ is called minimal and we denote $a \in \CM^*(\NN)$ (respectively $a\in \CA^*(\NN)$) if the sequence
$$\mbox{$\{a_0-\epsilon, a_1,\cdots,a_k,\cdots\}\quad$ (respectively $\;\,\{a_0, a_1-\epsilon,\cdots,a_k-\epsilon,\cdots\}$)}$$ is not completely monotone (respectively alternating) for any positive $\epsilon$.

\item Equivalently,  $a$ is minimal if and only if the measure $\nu$ in (\ref{bergu1}) (respectively $\mu$ in (\ref{bergu2})) has no point mass at zero.
\end{enumerate} \label{WA}\end{definition}
\begin{example} The sequence $a={\left((k+1)^{-1}\right)}_{k\geq 0}$ ceases to be completely monotone if $a_0=1$ is replaced by $a_0=1-\epsilon$, since
$$(-1)^n \Delta^na(0)=\frac{1}{n+1}-\epsilon, \quad n\in \NN.$$
The analogous constatation holds for the completely alternating sequence ${\left(1-(k+1)^{-1}\right)}_{k\geq 0}$ accordingly to Definition \ref{WA}.
\end{example}

After the above comments and considerations on minimal sequences, Theorem \ref{bergh} could be specified as follows: taking $\widetilde{\nu}$ and  $\widetilde{\mu}$ obtained as the image of the measures $\nu$ and $\mu$ on $(0,\infty)$ in (\ref{bergu1}) and (\ref{bergu2}) through  the obvious change of variable $u=e^{-x}$, we have:

\begin{theorem}\vspace{-0.2cm}

\begin{enumerate}[(a)]
\item \cite[Theorem 14b, p. 14]{widder}  and \cite[Theorem 1]{atha}  A positive  sequence $a={\left(a_k\right)}_{k\geq 0}$ is
obtained by interpolating a member of  $\CM[0,\infty)$ (respectively $\BF$) on $\NN$ if and only if $a$ belongs to $\CM^*(\NN)$
(respectively $\CA^*(\NN)$).

\item Equivalently, a sequence ${\left(a_k\right)}_{k\geq 0}$ belongs to $\CM^*(\NN)$ (respectively belongs to $\CA^*(\NN)$ and positive) if and only if there exist a unique finite measure $\widetilde{\nu}$ on $(0,\infty)$ (respectively a unique  triplet $(q,d,\widetilde{\mu})$ where $q,\,d \geq 0$ and the measure $\widetilde{\mu}$ satisfying $\int_{(0,\infty)} (1\wedge u) \, \widetilde{\mu}(\du)<\infty$), such that:
\begin{equation}
a_k=\int_{[0,\infty)} e^{-ku} \widetilde{\nu}(\du)\quad \mbox{\it \Big(resp.} \;\; a_k=q+d k+\int_{(0,\infty)} (1-e^{-ku}) \, \widetilde{\mu}(\du) \Big),\quad k\geq 0. \label{refor}
\end{equation}
\end{enumerate} \label{theorem3}\end{theorem}

It is clear that the subclass  $\CM^*(\NN)$ and the subclass of positive  sequences in $\CA^*(\NN)$ are convex cones.
\section{Linking functions and sequences of the completely and alternating type}\label{main}

In the spirit of Theorems \ref{bergh} and  \ref{theorem3}, a natural question is to ask whether the completely monotone (respectively Bernstein)  character of  function   $f$ is entirely recognized via its associated sequence ${\left(f(k)\right)}_{k}$. This constitutes a kind converse
of Hausdorff's moment characterization theorem \cite{haus} which is formulated in Theorem \ref{bergh} or \ref{theorem3}. A complete answer is given in the following two subsections.
\subsection{Complete  monotonicity property of functions is recognized by their restriction on $\NN$}
\begin{theorem} Let  $\Psi :[0,\infty)\longrightarrow [0,\infty)$ be a bounded  function.  Then, $\Psi$ is completely monotone if and
only if the two following  conditions hold:\vspace{-0.2cm}
\begin{enumerate}[(a)]
\item the function $\Psi$ has an holomorphic extension on $Re(z)>0$
and remains bounded there;

\item the sequence $\big(\Psi(k)\big)_{k\geq 0}$ is completely monotone
and minimal.
\end{enumerate}\label{theorem4}\end{theorem}
\begin{corollary}
A function $\Psi: (0,\infty)\longrightarrow \RP$  is completely monotone if and only if the following two conditions hold: for some (and hence all)  positive sequence ${(\epsilon_n)}_{n\geq 0}$ such that  $\epsilon_n\to 0$,
\vspace{-0.2cm}
\begin{enumerate}[(i)]
\item the function $\Psi$  has a holomorphic extension on $Re(z)>0$
and remains bounded on $Re(z)>\epsilon_n$;

\item the sequence ${\big(\tau_{\epsilon_n}\Psi(k)\big)}_{k\geq 0}={\big(\Psi(\epsilon_n +k)\big)}_{k\geq 0}$ completely monotone
and minimal.
\end{enumerate} \label{corollary1}\end{corollary}
\begin{corollary}
Two  completely monotone functions on $(0,\infty)$ coincide on the set of positive integers starting from a certain  rank if and only if they are equal. If one of them extends to $[0,\infty)$, then so does the other and they coincide on  $[0,\infty)$.
\label{corollary2}\end{corollary}
\subsection{Complete  monotonicity property of functions is recognized by their restriction on a lattices of the form $\alpha_n \NN$, where $\alpha_n\to 0$}
The following two results characterize complete monotonicity of functions  only  in terms of minimal completely monotone sequences,
i.e. condition (a) in Theorem \ref{theorem4} and Corollary \ref{corollary1} would be self contained.
\begin{proposition}A function $\Psi : [0,\infty) \to [0,\infty)$ belongs to $\CM[0,\infty)$ if and only if it is continuous and for some (and hence for all)  sequence ${\left(\alpha_n\right)}_{n\geq 0}$ of positive numbers tending to zero,  there corresponds a sequence ${\left(\Psi_n\right)}_{n\geq 0}$ in $\CM[0,\infty)$ such that the following representation holds each for each $n\in \NN$:
$$\Psi(\alpha_n\, k)=\Psi_n (k), \quad \mbox{for all}\;\,  k\in\NN\qquad \big(\mbox{\it i.e.}\; {\big(\Psi\big(\alpha_n k) \big)}_{k\geq 0} \in \CM^*(\NN)\big).$$
\label{proposition4}\end{proposition}
For non-bounded completely monotone functions on $(0,\infty)$ an analogous statement is given, but we require a minor correction consisting  on shifting the function on the right of zero:
\begin{corollary} For a  function $\Psi : (0,\infty) \to [0,\infty)$, the following conditions are equivalent: \vspace{-0.2cm}
\begin{enumerate}[\;(a)]
\item  $\Psi$ belongs to $\CM(0,\infty)$;

\item $\Psi$ is continuous\ and to every  sequence ${\left(r_n\right)}_{n\geq 0}$ of positive rational numbers tending to zero, there corresponds a  sequence ${\left(\Psi_n\right)}_{n\geq0}$ in $\CM[0,\infty)$,   such that following representation holds for each $n\in \NN$:
$$\Psi(r_n(k+1))=\Psi_n (k), \quad \mbox{for all}\;\, k\in\NN \qquad \big(\mbox{\it i.e.}\; {\big(\Psi\big(r_n(k+1)) \big)}_{k\geq 0} \in \CM^*(\NN)\big);$$

\item   $\Psi$ is continuous and there exists a sequence ${\left(\Psi_n\right)}_{n >0}$  in
$\CM[0,\infty)$,   such that the following representation holds for
each $n\in \No$:
$$\Psi(\frac{k+1}{n})=\Psi_n (k), \quad \mbox{for all}\;\, k\in\NN \qquad \big(\mbox{\it i.e.}\; {\big(\Psi\big(\frac{k+1}{n}) \big)}_{k\geq 0} \in \CM^*(\NN)\big).$$
\end{enumerate} \label{corollary3}\end{corollary}
\begin{remark} (i) By continuity, it is not difficult to see that assertions  in Proposition  \ref{proposition4} (respectively Corollary \ref{corollary3}) are also equivalent to the following:
\begin{quote}
 $\Psi$ is continuous and the sequence ${\big(\Psi(xk)\big)}_{k\geq 0}$ (respectively ${\Big(\Psi\big(x(k+1)\big)\Big)}_{k\geq 0}$) belongs to $ \CM^*(\NN)$ for every $x\in (0,\infty)$ or for every $x \in \mathbb{Q}\cap (0,\infty)$,
\end{quote}
The latter is precisely what is stated in Lemma 3.1 in \cite{mai} in case $\Psi(0)=1$, the minimality condition was somehow occulted.

(ii) The reader could notice that Theorem \ref{theorem4} requires a supplementary assumption of holomorphy and of boundedness compared to Proposition \ref{proposition4} and Corollary \ref{corollary3}. The point is that Theorem \ref{theorem4} gives more information since for every function $\Psi$ satisfying condition (a) therein, we have
\begin{eqnarray}
\Psi \in \CM(0,\infty) &\Longleftrightarrow& \sigma_x\Psi \in \CM(0,\infty),\;  \mbox{\it for some}\; x \in (0,\infty) \nonumber \\
&\Longleftrightarrow& {\big(\Psi\big(x(k+1)\big)\big)}_{k\geq 0}\in
\CM^*(\NN), \; \mbox{\it for some}\; x \in (0,\infty).\label{maiss}
\end{eqnarray}
The same holds for $\Psi \in \CM[0,\infty)$  under the additional condition of finiteness of $\Psi(0+)$. The condition of minimality and  holomorphy appear to be the lowest price to pay in order to have the condition (\ref{maiss}) expressed for a single $x$ instead of all $x$. \label{equivcm}\end{remark}
\subsection{Bernstein property of functions is recognized by their restriction on $\NN$}
\begin{theorem} A function $\Phi:[0,\infty)\to [0,\infty)$ is a Bernstein function if and only if it

\begin{enumerate}[(a)]
\item the function $\Phi$ has an holomorphic extension on $Re(z)>0$ and satisfies there $|\Phi(c+z)-\Phi(z)|\leq M$ for some $c,\,M>0$;

\item the sequence $\big(\Phi(k)\big)_{k\geq 0}$ is completely alternating and minimal.
\end{enumerate}
\label{theorem5}\end{theorem}
Since every Bernstein functions $\Phi$ satisfies $\lambda \mapsto \Phi(\lambda)/\lambda \in \CM(0,\infty)$, we immediately deduce from Corollary \ref{corollary2} the following:
\begin{corollary}
Two  Bernstein functions coincide on the set of non-negative integers starting from a certain  rank
if and only if they are equal on $[0,\infty)$.
\label{co5}\end{corollary}
\subsection{Bernstein property of functions is recognized by their restriction on lattices of the form $\alpha_n\NN$, where $\alpha\to 0$ }
As for completely monotone functions, the following two results characterize Bernstein property of functions  only  in terms of minimal completely alternating sequences, i.e. condition (a) in Theorem \ref{theorem5} would be self contained.
\begin{proposition} A  function $\Phi: [0,\infty)\longrightarrow\RP$ belongs to $\BF_b^0$ if and only if it is continuous and for some (and hence for all) sequence ${\left(\alpha_n\right)}_{n\geq 0}$ of positive numbers tending to zero,  there corresponds a sequence ${\left(\Phi_n\right)}_{n\geq
0}$ in $\BF_b^0$, such that the following representation holds for each $n\in \NN$:
$$\Phi(\alpha_n k)=\Phi_n(k),\quad \mbox{for all}\quad k\in\NN  \qquad \big(\mbox{\it i.e.}\; {\big(\Phi\big(\alpha_n k) \big)}_{k\geq 0} \in \CA^*(\NN)\big).$$
\label{proposition5}\end{proposition}
\begin{corollary} For a  function $\Phi: [0,\infty)\longrightarrow [0,\infty)$, the following conditions are
equivalent: \vspace{-0.2cm}
\begin{enumerate}[\;(a)]
\item $\Phi$ belongs to $\BF\,$;

\item $\Phi$ is continuous and to every  sequence ${\left(r_n\right)}_{n\geq 0}$ of positive rational numbers tending to zero,  there corresponds a sequence ${\left(\Phi_n\right)}_{n\geq 0}$ in $\BF$,   such that the following representation holds for each  $n\in \NN$:
$$\Phi(r_n \,k)=\Phi_n(k),\quad \mbox{for all}\quad k\in\NN \qquad \big(\mbox{\it i.e.}\; {\big(\Phi\big(r_n k) \big)}_{k\geq 0} \in \CA^*(\NN)\big);$$

\item $\Phi$ is continuous and there exists a sequence ${\left(\Phi_n\right)}_{n>0}$ in $\BF$,
such that the following representation holds for each $n\in \No$:
$$\Phi(\frac{k}{n})=\Phi_n(k),\quad \mbox{for all}\quad k\in\NN  \qquad \big(\mbox{\it i.e.}\; {\big(\Phi\big(\frac{k}{n}) \big)}_{k\geq 0} \in \CA^*(\NN)\big).$$
\end{enumerate} \label{corollary5}\end{corollary}
\begin{remark} As in Remark \ref{equivcm}, we can notice the following:

(i)  By continuity, assertions  in Corollary  \ref{corollary5}  (respectively Proposition \ref{proposition5}) are equivalent to the following assertion:
\begin{quote}
 $\Phi$ is continuous and the sequence ${\big(\Phi(x k)\big)}_{k\geq 0}$  belongs to $\in \CA^*(\NN)$ (respectively belongs to $\in \CA^*(\NN)$ and is bounded) for every $x \in (0,\infty)$ or for every  $x\in{\mathbb Q}_+$.
\end{quote}
(ii) Theorem \ref{theorem5} requires a supplementary assumption of holomorphy and of sub-affinity compared to Proposition  \ref{proposition5} and Corollary \ref{corollary5}. Theorem \ref{theorem5} gives more information since for every function $\Phi$ satisfying condition (a) therein, we have
\begin{quote}
$\Phi \in \BF \;\Longleftrightarrow \; \sigma_x\Phi \in \BF$ for some $x \in (0,\infty) \;\Longleftrightarrow \; {\Big(\Phi(x k)\Big)}_{k\geq 0}$  belongs to $\in \CA^*(\NN)$ (respectively belongs to $\in \CA^*(\NN)$ and is bounded) for some $x \in (0,\infty)$.
\end{quote}
\end{remark}
\section{Some pre-requisite}\label{prerequisite}
The following results are crucial in order to conduct our proofs.\\

\subsection{On iterative functional equations and asymptotic of differences}
We  first present a result of Webster \cite{Webster-1997} which will be used in the proofs of Propositions \ref{proposition1} and \ref{proposition2}. Given a log-concave function $g:[0,\infty)\rightarrow [0,\infty)$, he considered the iterative functional equation
\begin{equation}\label{functional-equation}
f(x+1)=g(x)f(x),\quad  x>0, \quad \hbox{and}\quad f(1)=1.
\end{equation}
Motivated by the study of generalized gamma functions and their characterization by a Bohr-Mollerup-Artin type theorem, Webster studied equations  of type (\ref{functional-equation}). A combination of Theorems 4.1 and 4.2 \cite{Webster-1997} gives results that were stated in \cite{ajr} under this form:
\begin{theorem}\label{Webster}[Webster, \cite{Webster-1997}]  Assume that $g:[0,\infty)\to [0,\infty)$ is log-concave and $ \lim_{a\rightarrow \infty} \frac{g(x+a)}{g(a)}=1$ for every $a>0$. For $n\geq 1$, let $a_n= (g'_{-}(n)+g'_{+}(n))/2g(n)$ and  $\gamma_g= \lim_{n\rightarrow \infty} \left(\sum_1^{n} a_j-\log g(n) \right)$. Then, there exists a unique  log-convex  solution $f:[0,\infty)\rightarrow [0,\infty)$  to the functional equation
(\ref{functional-equation}) satisfying $f(1)=1$ and given by
\begin{equation}\label{gamma-1}
f(x)=\frac{e^{-\gamma_g x}}{g(x)}
\prod_{n=1}^{\infty}\frac{g(n)}{g(n+x)}e^{a_n x},\quad x>0.
\end{equation}
If furthermore $ \lim_{a\rightarrow \infty}  g(x)=1$, then the representation simplifies to
\begin{equation}\label{gamma-2}
f(x)=\frac{1}{g(x)} \prod_{n=1}^{\infty}\frac{g(n)}{g(n+x)},\quad x>0.
\end{equation}
\end{theorem}

Theorem 1.1.8 p. 5 \cite{bin} says that if $l:\RR\to \RR$   is additive (i.e. $l(x+y)=l(x)+l(y),  \; \forall x,y \in \RR$), and measurable, then $l(x) = Cx$ for some  $C\in\RR$. On the other hand, consider a function $l:[0,\infty)\to [0,\infty)$ solution of the iterative equation
$$l(x+1)=l(x)+l(1),  \quad x \in (0,\infty).$$
Take $g(x)=e^{l(1)}$ and $f(x)=e^{l(x)-l(1)}$ in Theorem \ref{Webster}. Clearly, $a_n=0$  and $\gamma_g=-l(1)$ and (\ref{gamma-1}) yields that the unique convex solution is given by $l(x) = l(1)\,x$, $x\geq 0$. It would be desiderate to have a similar conclusion without the convexity assumption. Karamata's characterization theorem for regularly varying functions (Theorem 1.4.1 p.17 in \cite{bin}), says that if $\lim_{x\to \infty}h(\lambda +x)-h(x) =l(\lambda)$, then there exists a real number $\rho$ such that $\lim_{x\to \infty}\big(h(\lambda +x)-h(x)\big)=\rho \lambda$ for every $\lambda\geq 0$. We propose the following lemma as an improvement of Karamata's characterization:
\begin{lemma} Suppose two function $h, l:[0,\infty)\to [0,\infty)$ are linked for every $\lambda \geq 0$ by
the limit
 $$h(\lambda+n)-h(n)\to l(\lambda),   \quad \mbox{as $n\to \infty$ and $n\in \No$}.$$
Then, necessarily  $l(\lambda)=\lambda l(1)$ with  $l(1)\geq 0$ and
\begin{equation} h(\lambda+x)-h(x)\to l(\lambda),  \quad \mbox{as $x\to \infty$, uniformly in each compact $\lambda$-set in $[0,\infty)$}
\label{conv} \end{equation}
\end{lemma}
\begin{proof} The proof goes through the following four steps:

\noindent a) For every $\lambda \geq 0$, write  that
$$l(\lambda+1)=\lim_{n\to \infty}[h(\lambda+1+n)-h(n)]=\lim_{n\to \infty}[h(\lambda+1+n)-h(n+1)] +\lim_{n\to \infty}[h(n+1)-h(n)] =  l(\lambda)+l(1)$$
and  retrieve that
\begin{equation}\label{lmn}
l(\lambda+m)=l(\lambda)+ l(m)=l(\lambda)+ l(1)\, m,\quad \mbox{for every}\; \lambda \geq 0, m\in\NN.
\end{equation}
Since $h(n+1)-h(n)$ converges to $l(1)$, then, so does its Ces\`aro mean
$$ l(1)= \lim_{n\to \infty} \frac{1}{n}\sum_{i=0}^{n-1}[h(i+1)-h(i)]=\lim_{n\to \infty} \frac{h(n)-h(0)}{n}=\lim_{n\to \infty} \frac{h(n)}{n}, $$
and deduce that $l(1)\geq 0$.

\noindent b) Case where $l\equiv 0$ (i.e. $l\equiv0$): Assume that a function  $k:[0,\infty)\to [0,\infty)$ satisfies
\begin{equation*}\label{kmn}
\lim_{n\to \infty, \, n\in \No}k(\lambda+n)-k(n)\to 0.
\end{equation*}
Reproduce identically the first proof of Theorem 1.2.1 p. 6 \cite{bin} (by taking with their notations $x=n\in \No$) in order to get $k(\lambda+n)-k(n)\to 0$ uniformly in each compact $\lambda$-set in $(0,\infty)$ as $n\to \infty$ and $n\in \No$. Denote $\{x\}$ and $[x]$ the fractional and integer part of $x$. Then, mimicking  the end of the second proof of Theorem 1.2.1 p. 6 \cite{bin}, take an arbitrary  compact interval $[a,b]$ in $[0,\infty)$ and  observe that
\begin{eqnarray*}
\sup_{\lambda\in [a,b]}\big|k(\lambda+x)-k(x)\big|&=&\sup_{\lambda\in [a,b]}\big|k(\lambda+\{x\}+[x])-k(\{x\}+[x])\big|\\
&\leq& \sup_{u\in [a,b+1]} \big|k(u+[x])-k([x])\big|+\sup_{u\in [0,1]}\big|k(u+[x])-k([x])\big|\\
\end{eqnarray*}
goes to zero as $[x]\to \infty$. Finally, get
\begin{equation} k(\lambda+x)-k(x)\to 0,  \quad \mbox{as $x\to \infty$, uniformly in each compact $\lambda$-set in $[0,\infty)$}.
\label{conv0} \end{equation}

\noindent c) Case where $l \equiv\!\!\! \!\!/$ $0$: Taking $k(x)=h(x)-l(x)$ and
using (\ref{lmn}),  obtain for every $\lambda>0$
$$k(\lambda+n)-k(n)=  h(\lambda+n) - l(\lambda+n) -h(n) + l(n) =  h(\lambda+n) - h(n)-l(\lambda) \to 0.$$
as $n\to \infty$. By step b)  deduce that $k$ satisfies (\ref{conv0}).

\noindent d) Taking $h(x)=\log f(e^x)$ with $f$ as in Theorem 1.4.1 p. 17 \cite{bin}, conclude that necessarily the function $l$ is linear, i.e. $l(\lambda)=l(1) \lambda$.
\end{proof}
\subsection{On Blaschke's characterization theorem}
The second result, due to Blaschke, allows to identify holomorphic functions given their restriction along suitable sequences:
\begin{theorem}[Blaschke,  Corollary  p. 312 in Rudin \cite{rudin}]
If $f$ is holomorphic and bounded on the open unit disc $D$,  if $\alpha_1,\alpha_2,\alpha_3, \cdots$  are the zeros of $f$ in $D$ and if $\sum_{i=1}^\infty  (1 - |\alpha_i|) = \infty,$ then $f(z) = 0$ for all $z \in D$.
\end{theorem}
Using the is conformal one-to-one mapping of the open unit disc onto the open right half plane
$$\theta(z) = \frac{1+z}{1-z},$$
one can easily rephrase  Blaschke's theorem for function defined on the open right half plane:
\begin{corollary}\label{blas}
Two holomorphic functions on the open right half plane $P$ are identical if their difference is bounded and they coincide along a sequence $z_1,z_2,z_3, \cdots$ in $P$,  such that the series $\sum   (1 - |\frac{z_i-1}{z_i+1}|)$ diverge and in particular for $z_i=i\in \No.$
\end{corollary}
\begin{remark}
Corollary \ref{blas} will be used essentially in the proofs of Theorems \ref{theorem4} and \ref{theorem5} for checking the equality between two functions coinciding along the sequence of positive integers. We are totally aware that Theorems \ref{theorem4} and \ref{theorem5} could be rephrased in a more general setting with different sequences. For clarity's sake, we preferred to state our results there under their current form.
\end{remark}
\subsection{On Gregory-Newton development}
In the alternative proofs of Theorems \ref{theorem4} and \ref{theorem5}, we will also need the concept of Gregory-Newton development that we recall here:
\begin{definition} A function $f$ defined on some domain $D$ of the complex plane is said to admit a Gregory-Newton development if there exists some sequence ${(a_k)}_{k\geq 0}$ such that
\begin{equation*}\label{newton}
f(z)=\sum_{k=0}^{\infty}(-1)^k \, \frac{a_k}{k!} \,  z^{\underline{k}},\quad z\in D,
\end{equation*}
where
$$z^{\underline{0}} = 1\quad \mbox{and}\quad z^{\underline{k}}= z(z-1)\cdots(z-k+1)=1,\;\; k\geq 1.$$
\end{definition}
\begin{remark} (i) Notice that the factorial powers $z^{\underline{n}}$ and the usual powers $z^k$ are related through the relations
$$z^{\underline{n}} = \sum_{k=0}^n {n\brack k}\,(-1)^{n-k}\, z^k \quad \mbox{and} \quad z^n = \sum_{k=0}^n {n\brace k}\, z^{\underline{k}}\;,$$
where ${ n\brack k}$ and ${n\brace k}$ are the Stirling numbers of the first and second kind respectively. These relations allow to swap between
Gregory-Newton and power series developments  whenever it is possible. This clarifies why a Gregory-Newton development for a holomorphic function is unique.

(ii) For functions $f$ admitting a Gregory-Newton development, N\"{o}rlund (\cite{nor}  p. 103), showed that necessarily
$$a_k= (-1)^k \, \Delta^k f(0),\quad k\geq 0.$$

(iii) It is worth noting that  the transformation
$${\big(f(l)\big)}_{l=0, \cdots m}\mapsto {\big((-1)^n\,\Delta^n f(0)\big)}_{n=0, \cdots m}$$ is the classical binomial transform which is involutive. Since the operators $\tau$ and $\Delta$ commute, and so do their iterates, it is immediate that the transformation
${\big(f(k+l)\big)}_{l=0, \cdots m}\mapsto {\big((-1)^n\,\Delta^n f(k)\big)}_{n=0, \cdots m}$ is also involutive for every fixed $k\in \NN$. The transformation ${\big(f(l)\big)}_{l=0, \cdots m}\mapsto {\big(\Delta^n f(0)\big)}_{n=0, \cdots m}$ is called the Euler transform. It is not an  involution but remains one-to-one (see \cite{flageolet}). It is now clear that
\begin{equation}\label{una}
\mbox{the sequence}\;\, {\left(\Delta^k f(0)\right)}_{k\geq 0} \;\, \mbox{is one-to-one with the sequence}\;\, {\left(f(k)\right)}_{k\geq 0}.
\end{equation}
\label{1a1}\end{remark}

It is trivial that any function $f:D\subset \mathbb{C}\to \mathbb{C}$ could be represented by an interpolating polynomial $P_n$ of a degree $n\geq 1$, plus a remainder function $R_n$:
$$f=P_n+R_n,\quad\mbox{where} \quad P_n(z)=\sum_{k=0}^{n}  \frac{\Delta^k f(0)}{k!} \, z^{\underline{k}}\,.$$
The following result clarifies when the remainder function goes to zero, i.e. when $f$ could be expanded  in a unique way (see point (i) in Remark \ref{1a1}) into a Gregory-Newton series given by
\begin{equation}\label{newton}
f(z)=\sum_{k=0}^{\infty}\frac{\Delta^k f(0)}{k!} \,
z^{\underline{k}}.
\end{equation}
\begin{theorem}[N\"{o}rlund, \cite{nor}  p. 148] In order that a function $f$ admits a Gregory-Newton development (\ref{newton}), it is necessary and sufficient that $f$ is holomorphic in a certain half-plane $Re(z)>\alpha$   and  $f$ is of the exponential type, i.e.
\begin{equation}\label{nor}
\Big|f(z)\Big| \leq Ce^{D |z|},
\end{equation}
where $C$ and $D$ are fixed positive numbers.
\label{norlund}\end{theorem}

As an application, we propose the following:
\begin{proposition} 1) Every  bounded completely monotone function   $\Psi$ admits an extension  which

(i) is bounded, continuous on the half plane $Re(z)\geq 0$ and  holomorphic  on $Re(z)>0$;

(ii)  is expandable into a Gregory-Newton series on the half plane $Re(z)>0$.

\noindent 2) Every   Bernstein function $\Phi$ admits an extension  which

(i) is continuous on the half plane $Re(z)\geq 0$ and  holomorphic  on $Re(z)>0$;

(ii) satisfies for some  $C,\, D\geq 0$
\begin{equation*}
|\Phi(z)-\Phi(z')| \leq  C+ D\,|z-z'| \quad \mbox{for every}\; z,\,z' \; s.t.\; Re(z)\geq Re(z')\geq 0; \label{norb}
\end{equation*}

(ii)  is expandable into a Gregory-Newton series on the half plane $Re(z)>0$.
\label{holom}\end{proposition}
\begin{proof}  1) Assertion (i) is due to Corollary 9.12 p. 67 \cite{berg2}. Boundedness of the extension of $\Psi$ insures that N\"{o}rlund's condition (\ref{nor}) is satisfied and then (ii) is true.

\noindent 2) Assertion (i) is due to 9.14 p. 68 \cite{berg2} or to Proposition 3.6 p. 25 \cite{SSV}, so that the representation (\ref{berep}) extends on $Re (z) \geq  0$
$$\Phi(z)=q+ d\,z +\int_{(0,\infty)}(1-e^{-z x})\mu(\dx).$$
For 2)(ii), we reproduce some steps   of the Proposition 3.6 p. 25 \cite{SSV}, we observe that for every $x \geq  0$ and $z,\,z'\in
\mathbb{C}$ such that $Re (z)\geq Re (z')\geq 0$, we have
$$|e^{-z x}-e^{-z' x}|\leq |1-e^{-(z-z') x}| \leq 2\wedge |(z-z') x| \leq (2\vee|z-z'|)(1\wedge x)\leq (2+|z-z'|)(1\wedge x).$$
We deduce
\begin{eqnarray*}
|\Phi(z)-\Phi(z')| &\leq& d\,|z-z'| + \int_{(0,\infty)} |e^{-z x}-e^{-z' x}| \mu(\dx)\\
&\leq & d|z-z'| +  (2+|z-z'|)\int_{(0,\infty)} (1\wedge x)\mu(\dx) \\
&=&  C+ D\,|z-z'| \leq  (C\vee D)\, e^{D|z-z'|}.
\end{eqnarray*}
where $C= 2 \int_{(0,\infty)} (1\wedge x)\mu(\dx)$ and $D=d+\int_{(0,\infty)} (1\wedge x)\mu(\dx)$.

2)(iii) is justified as follows: take $z'=0$,  get that $|\Phi(z)|\leq |\Phi(0)|+C+ D\,|z| \leq  \big(C+|\Phi(0)|\big)\vee D)\, e^{D|z|}$ and deduce $\Phi$ satisfies N\"{o}rlund's condition (\ref{nor}).
\end{proof}
\section{The proofs}\label{proof}
\begin{proof}[Proof of Proposition~\ref{proposition1}] (a) For the necessity part, notice that if $c>0$ and $\Psi$ is represented by $\Psi(\lambda)=\int_{(0,\infty)}e^{-\lambda x} \, \mu_\Psi(\dx),\;\lambda >0$, then
$$h(\lambda) := \Psi(\lambda)-\Psi(\lambda+c) =\int_{(0,\infty)}e^{-\lambda x}\,(1-e^{-c x}) \, \mu_\Psi(\dx)$$
since the measure $\mu_h(\dx):=(1-e^{-c x}) \, \mu_\Psi(\dx)$ gives no mass to zero.

For the sufficiency part, take $c>0$ and consider the iterative functional equation $\Psi(\lambda)-\Psi(\lambda+c)=h(\lambda)$ with $h\in \CM(0,\infty)$ represented by
$$h(\lambda)=\int_{(0,\infty)}e^{-\lambda x} \, \mu_h(\dx).$$ We would like to show that $\Psi \in \CM(0,\infty)$, or equivalently (by Remark \ref{clos} (ii)) that $\sigma_c\Psi \in \CM(0,\infty)$. This is the reason why it is sufficient to show that the solution of the iterative functional equation $$\Psi(\lambda)-\Psi(\lambda+1)=h(\lambda)$$
belongs to $\CM(0,\infty)$, i.e. to check things with $c=1$. For this purpose, we  apply Theorem \ref{Webster} with the log-concave function $g(\lambda)=e^{-h(\lambda)},\;\lambda >0$ satisfying $\lim_{\lambda \to \infty}g(\lambda)=1$ and $f(\lambda)=e^{\Psi(\lambda)-\Psi(1)},\;\lambda >0$. We obtain the representation:
$$\Psi(\lambda)-\Psi(1)=h(\lambda) -\sum_{n=1}^\infty \int_{(0,\infty)}e^{-n x}\, (1-e^{-\lambda x})\mu_h(\dx)=\int_{(0,\infty)}  (e^{-\lambda x}-e^{- x})\mu_h(\dx),$$
which insures that $\Psi$ is differentiable with $-\Psi'\in \CM(0,\infty)$. Because $\Psi$ is non-negative, we conclude that  $\Psi\in \CM(0,\infty)$.

Statement (b) could be extracted from the second proof that follows.
\end{proof}
\begin{proof}[Second proof of the sufficiency part of Proposition~\ref{proposition1}]  Fix $c>0$ and  write for every $n\in
\No$ and $\lambda >0$,
$$(-\Delta_{nc})\Psi(\lambda)=\Psi(\lambda) - \Psi(\lambda +nc)= \sum_{i=0}^{n-1}\Psi(\lambda+ic) - \Psi(\lambda +(i+1)c)= \sum_{i=0}^{n-1}(-\Delta_c)\Psi(\lambda+ic)$$
Obviously, the sequence $n \mapsto (-\Delta_{nc})\Psi(\lambda)$ is increasing for every $\lambda, \, c>0$, then  $x\mapsto \Psi(x)$ is decreasing and then converging, since non-negative. We denote $\Psi(\infty):=\lim_{x\to
\infty} \Psi(x)$. The function $\lambda \mapsto (-\Delta_{nc})\Psi(\lambda)$ belongs
to $\CM(0,\infty)$ and,  by Corollary 1.7 p. 6 in \cite{SSV}, the limiting function $(-\Delta_{\infty,c})\Psi:=\lim_{n\to \infty}
(-\Delta_{nc})\Psi$ also belongs to $\CM(0,\infty)$, the convergence holds locally uniformly and also for the derivatives. This limit does not depend on $c>$ since it satisfies:
$$\Psi(\lambda) =\Psi(\infty) +(-\Delta_{\infty,c})\Psi(\lambda),\quad \lambda >0.$$
\end{proof}
\begin{proof}[Proof of Proposition~\ref{proposition2}] 1) If $\Phi \in \BF$ is represented by (\ref{berep}), then for every  $c>0$,
$$\lambda\mapsto  \Delta_c\Phi(\lambda)= \Phi(\lambda+c)- \Phi(\lambda)= dc+\int_{(0,\infty)} e^{-\lambda x} (1-e^{-c
x})\mu(\dx),\quad\lambda \geq 0,$$
is non-negative and belongs to $\CM[0,\infty)$. By Remark \ref{r1} (iii) we deduce that $\Phi\in \CA[0,\infty)$.

Conversely, assume $\Phi \in \CA[0,\infty)$ and non-negative, we will show that $\Phi$ is differentiable and that $\Phi'$ in completely monotone on $(0,\infty)$ which is equivalent to $\Phi \in \BF$. Remark \ref{r1} (iii) and definiteness of $\Phi$ in zero yield to
$$\lambda\mapsto  \Delta_c\Phi(\lambda)= \Phi(\lambda+c)- \Phi(\lambda) \in \CM[0,\infty),\quad \forall c>0.$$
Inspired by the proof of Proposition~\ref{proposition1}, we will see, that  $\Delta\Phi \in \CM(0,\infty)$ (i.e. when taking $c=1$) is sufficient for proving that $\Phi$ is differentiable and that $\Phi'\in \CM(0,\infty)$. Indeed, assume $\Delta\Phi$ is the Laplace representation $\mu$
$$\Delta\Phi(\lambda)=\int_{[0,\infty)}e^{-\lambda x} \mu(\dx).$$
Theorem \ref{Webster} insures that $f(\lambda)=e^{\Phi(1)-\Phi(\lambda)}$ is the unique solution of the iterative functional equation $f(\lambda+1) =f(\lambda) g(\lambda)$ and $\Phi(1) -\Phi(\lambda)$ has the following representation for every $\lambda >0$:
\begin{eqnarray*}
\Phi(1) -\Phi(\lambda)&=&\Delta\Phi(\lambda)  -\sum_{n=1}^\infty \Delta\Phi(n)-\Delta\Phi(n+\lambda)\\
&=&\int_{[0,\infty)}e^{-\lambda x}\,\mu(\dx)-\sum_{n=1}^\infty \int_{(0,\infty)}(e^{-n x}- e^{-(n+\lambda) x})\,\mu(\dx)\\
&=&\mu(\{0\})+\int_{(0,\infty)}e^{-\lambda x}\,\mu(\dx)- \int_{(0,\infty)}  (1-e^{-\lambda x})\,\frac{e^{-x}}{1-e^{-x}}\, \mu(\dx)\\
&=&\mu(\{0\})+\int_{(0,\infty)} \frac{e^{-x}-e^{-\lambda x}}{1-e^{-x}} \mu(\dx).
\end{eqnarray*}
Then, for every $a>0$, $\lambda \mapsto \Phi(\lambda +a)-\Phi(a)=\int_{(0,\infty)}(1-e^{-\lambda x})\,\frac{e^{-ax}}{1-e^{-x}}\,\mu(\dx)$
is a Bernstein function  which is equivalent, by (\ref{pc}) to $\Phi \in \BF$.

2) The proof is conducted identically by dropping the positivity condition on $\Phi$.
\end{proof}
\begin{proof}[Proof of Proposition~\ref{proposition3}]  If $\Phi \in \BF$ is represented by (\ref{berep}) and if $c>0$, then
$$\lambda \mapsto \theta_c\Phi(\lambda)= \Phi(c)- \Phi(0)+\Phi(\lambda)-\Phi(\lambda+c)=\int_{(0,\infty)}(1-e^{-\lambda x})(1-e^{-cx})\mu(\dx).$$
We deduce that $\theta_c\Phi \in \BF_b^0$ since $(1-e^{-c x})\mu(\dx)$ is a measure with finite total mass equal to
$\Phi(c)-\big(\Phi(0) + d\,c\big)$.

Conversely, assume  $\lambda \mapsto \theta_c\Phi(\lambda)= \big[\Phi(c)- \Phi(0)\big]-\big[\Phi(\lambda+c)-\Phi(\lambda)]\in \BF_b^0$. The latter is equivalent by (\ref{bcm}) to $\lambda \mapsto \big[\Phi(\lambda+c)-\Phi(\lambda)]\in \CM[0,\infty)$ and we conclude as in the proof of Proposition~\ref{proposition2}.

Statement (b) could be extracted from the second proof that follows.
\end{proof}
\begin{proof}[Second proof of of the sufficiency part of Proposition~\ref{proposition3}]  Because of the invariance (\ref{bc}), it is enough to prove the Proportion in case where $c=1$. Since $\theta\Phi$ belongs to $\BF_b^0$, then it is represented with a finite measure $\mu $ on $(0,\infty)$ by
\begin{equation*}\label{d0}
\theta_c\Phi(\lambda)= \int_{(0,\infty)}(1-e^{-\lambda x})\,\mu_c(\dx),\quad \lambda \geq 0.
\end{equation*}
We will see that the latter is sufficient to show that $\phi$ is differentiable on $(0,\infty)$ and that $\Phi'$ belongs to $\CM(0,\infty)$. First notice that for every $n\in \NN$ and $\lambda \geq 0$,
\begin{eqnarray*}
\theta_{nc}\Phi(\lambda)&=& \big[\Phi(nc) -\Phi(0)\big]-\big[\Phi(\lambda +nc)-\Phi(\lambda)\big]\\
&=&\sum_{k=0}^{n-1}\big[\Phi((k+1)c) -\Phi(kc)\big]-\big[\Phi(\lambda +(k+1)c)-\Phi(\lambda+kc)\big]\\
&=&\sum_{k=0}^{n-1}\Big\{\big[\Phi(c) -\Phi(0)\big]-\big[\Phi(\lambda +(k+1)c)-\Phi(\lambda+kc)\big]\Big\}\\
&& \qquad\qquad -\Big\{\big[\Phi(c) -\Phi(0)\big]-\big[\Phi((k+1)c)-\Phi(kc)\big]\Big\}\\
&=&\sum_{k=0}^{n-1} \theta_c\Phi(\lambda+kc) -\theta_c\Phi(kc)\\
&=&\sum_{k=0}^{n-1}\int_{(0,\infty)}(1-e^{-\lambda x})\, e^{-kc x}\,\mu_c(\dx)\\
&=&\int_{(0,\infty)}(1-e^{-\lambda x})\,\frac{1-e^{-nc x}}{1-e^{-c x}}\mu_c(\dx).
\end{eqnarray*}
By Corollary 3.9 p. 29 \cite{SSV}, the sequence $\theta_{nc}\Phi$ converges locally uniformly, and all its derivatives to a Bernstein function $\theta_{\infty,c}\Phi$ given by
$$\lambda \mapsto \theta_{\infty,c}\Phi(\lambda)=\int_{(0,\infty)} \,\frac{1-e^{-\lambda x}}{1-e^{- x}} \mu_c(\dx)\in\BF.$$
We have also showed that for every $\lambda\geq 0$,
$\Phi(nc)-\Phi(\lambda +nc)\to  \theta_{\infty,c}\Phi(\lambda)
+\Phi(0)-\Phi(\lambda)$, when $n\to \infty$. On the other hand, by (\ref{conv}), we get that for every $\lambda\geq 0$
$$\lim_{x\to \infty}\Phi(\lambda +x) -\Phi(x)= d_c\lambda, \quad \mbox{for some}\; d_c\geq 0$$
and  we deduce that,
$$ \Phi(\lambda)=   \Phi(0)+d_c \lambda+\theta_{\infty,c}\Phi(\lambda), \quad \lambda\geq 0.$$
Unicity of the triplet of characteristics in the representation (\ref{berep}) of Bernstein functions allows to conclude that $d_c=\lim_{x\to \infty}\Phi(x)/x$ and  $\theta_{\infty,c}$, both  do not depend on $c$.
\end{proof}
\begin{proof}[Proof of Theorem~\ref{theorem4}] For the necessity part, use Proposition \ref{holom} for (a) and  Theorem \ref{theorem3}  for (b).
For the sufficiency part, use Theorem \ref{theorem3} again which asserts that there is a unique finite measure $\mu$ on $[0,\infty)$
such that
$$\Psi(k)=\int_{[0,\infty)} e^{-k x} \,\mu(\dx),\quad \forall k\in\NN.$$
The finiteness of each term $\Psi(k),\,k\in \NN$ allows to define the function
$$\overline{\Psi}(\lambda):=\int_{[0,\infty)} e^{-\lambda x} \,\mu(\dx),\quad \lambda\geq 0.$$
Since   $\overline{\Psi}(k)=\Psi(k)$ for every $k\in \NN$, and since the extensions on $Re(z)>0$ of both functions $\Psi$ and $\overline{\Psi}$  are holomorphic and bounded, then Blaschke's argument given in Theorem \ref{blas} insures that the extensions of $\Psi$ and  $\overline{\Psi}$   are equal on  $Re(z)>0$. We deduce that  $\Psi$ and  $\overline{\Psi}$ coincide on $(0,\infty)$ and, by continuity in zero, also on $[0,\infty)$.
\end{proof}
\begin{proof}[Alternative Proof of Theorem~\ref{theorem4}]  We conclude as in the last proof without the use of Blaschke's argument.  Because the extensions on $Re(z)>0$ of both functions $\Psi$ and $\overline{\Psi}$  are holomorphic and bounded, they are, by Proposition \ref{holom} expandable into  Gregory-Newton series as in (\ref{newton}).  Since ${\left(\Psi(k)\right)}_{k\geq 0}={\left(\overline{\Psi}(k)\right)}_{k\geq 0}$ and  the sequences ${\left(\Delta^k\Psi(0)\right)}_{k\geq 0}$ and ${\left(\Psi(k)\right)}_{k\geq 0}$ entirely determine each other by (\ref{una}), we conclude that $\Delta^k\Psi(0)=\Delta^k\overline{\Psi}(0)$ for all $k\in \NN$. Finally, $\overline{\Psi}$ and $\Psi$ have the same expansion (\ref{newton}) and then are equal.
\end{proof}
\begin{proof}[Proof of Corollary~\ref{corollary1}] For the necessity part, do as in the proof of Theorem~\ref{theorem4}. For the sufficiency part, notice that the sequence of functions  $\tau_{\epsilon_n}\Psi(\lambda)=\Psi(\epsilon_n+\lambda) ,\, \lambda \geq 0$, satisfy the conditions of Theorem \ref{theorem4} and converge to $\Psi$.
One concludes with Remark \ref{clos} (ii).
\end{proof}
\begin{proof}[Proof of Corollary~\ref{corollary2}] The necessity part is obvious. For the sufficiency part, consider two functions $\Psi_1$ and $\Psi_2$ in $\CM(0,\infty)$, represented by their measures $\nu_1$ and $\nu_2$, and coinciding on $\{n_0,n_0+1,\cdots\}$ for some $n_0\in \NN$. By construction, the well defined functions on $[0,\infty)$,  $\tau_{n_0}\Psi_1(\lambda)$ and $\tau_{n_0}\Psi_2$, coincide on $\NN$. Using Remark \ref{1a1} and imitating the end of the proof of Theorem \ref{theorem4},  conclude that $\tau_{n_0}\Psi_1$ and $\tau_{n_0}\Psi_2$  are equal, that is $$\int_{[0,\infty)}e^{-\lambda x}\, e^{-n_0\,x} \nu_{1}(\dx)= \int_{[0,\infty)}e^{-\lambda x}\,e^{-n_0\,x}\nu_{2}(\dx),\quad \forall \lambda \geq 0$$
By injectivity of Laplace transform, conclude that the measures $e^{-n_0\,x} \nu_{1}(\dx)$ and $e^{-n_0\,x} \nu_{2}(\dx)$ are equal and so are $\nu_{1}$ and $\nu_{2}$. One can also use   the  Gregory-Newton expansion argument as in the alternative proof of Theorem~\ref{theorem4}. Now, assume $\Psi_1 (0+)<\infty$ (that is $\Psi_1 \in \CM[0,\infty)$), then, by continuity, necessarily $\Psi_1 (0+)=\Psi_2 (0+)$ and $\Psi_1= \Psi_2 $ on $[0,\infty)$.
\end{proof}
\begin{proof}[Proof of Proposition~\ref{proposition4}] The necessity part is obvious by Remark \ref{clos} (ii), we tackle the sufficiency part. Using continuity in zero, it is enough to prove that $\Psi$ is completely monotone on $(0,\infty)$. We fix $\lambda>0$ and  denote $[x]$ the integer part of the real number $x$. Notice that $\alpha_n[\frac{\lambda}{\alpha_n}]$ is smaller than $\lambda$ and tends to $\lambda$ when $n$ goes to infinity. We claim that
\begin{equation}\label{uniform}
e_n(\lambda, u):= e^{-\alpha_n[\frac{\lambda}{\alpha_n}]\,u}\longrightarrow e^{-\lambda\,u}, \;\,\mbox{uniformly in $u\geq 0$, when $n\to \infty$}.
\end{equation}
Indeed, using the inequality
$$a\,e^{-a} \leq 1 \quad\mbox{and}\quad 0\leq e^{-a}-e^{-b}=\int_a^b e^{-u} \,du\leq (b-a)\, e^{-a},\quad 0\leq a \leq b,$$
we have, for every integer $n$  such that $\alpha_n <\lambda$ and $u\geq 0$, that
$$0\leq e_n(\lambda, u)-  e^{-\lambda\,u} \leq \alpha_n\, u \,\left(\frac{\lambda}{\alpha_n}-[\frac{\lambda}{\alpha_n}] \right)\, e^{- \alpha_n [\frac{\lambda}{\alpha_n}] \,u} \leq \alpha_n\, u \, e^{- \alpha_n [\frac{\lambda}{\alpha_n}] \,u}\leq \frac{1}{[\frac{\lambda}{\alpha_n}]} \leq
\frac{\alpha_n}{\lambda-\alpha_n }.$$

Now, by assumption, we have
$$\Psi\Big(\alpha_n[\frac{\lambda}{\alpha_n}]\Big) =  \Psi_n\Big([\frac{\lambda}{\alpha_n}]\Big)=\int_{[0,\infty)} e^{-[\frac{\lambda}{\alpha_n}]u}\, \nu_n(\du)= \int_{[0,\infty)} e_n(\lambda, v)\, \widetilde{\nu}_n(\dv), $$
where $\nu_n$ is the representative measure of $\Psi_n$ and  $\widetilde{\nu}_n$ is the finite measure with total mass $\widetilde{\nu}_n\big(\big[0,\infty))=\Psi(0)$,  image of $\nu_n$ by  the change of variable $u=\alpha_n v$. Continuity of $\Psi$ yields
$$\Psi(\lambda)=\lim_{n\to \infty}\Psi\Big(\alpha_n[\frac{\lambda}{\alpha_n}]\Big) = \lim_{n\to \infty} \int_{[0,\infty)} e_n(\lambda, u)\,
\widetilde{\nu}_n(\du)$$ and  Helly's selection theorem, insures that there exist a subsequence $\left(\widetilde{\nu}_{n_p}\right)_{p\geq 0}$ and a finite measure $\nu$ on $[0,\infty)$ such that $\widetilde{\nu}_{n_p}$ converges vaguely (and also weakly) to $\nu$. Taking the limit along the subsequence $n_p$ and thanks to  the uniformity in (\ref{uniform}), we get
$$\Psi(\lambda)=\int_{[0,\infty)} e^{-\lambda\,u}\, \nu(\du).$$
\end{proof}
\begin{proof}[Proof of Corollary~\ref{corollary3}] Since  $(a) \Longrightarrow (b)$ is justified by Remark \ref{clos} (ii) and $(b) \Longrightarrow (c)$ is immediate, we just need to prove $(c) \Longrightarrow (a)$. In case where $\Psi(0+)<\infty$,  Proposition \ref{proposition4} directly applies. In case where $\Psi(0+)=\infty$, we claim that for every fixed $m\in\No$, the function
$$\tau_{\frac{1}{m}}\Psi(\lambda) =\Psi(\frac{1}{m}+\lambda), \;\lambda \geq 0,$$
satisfies the condition  of Proposition \ref{proposition4}. Indeed, $\tau_{\frac{1}{m}}\Psi$  is continuous and, by assumption, there exists
for each $n\in \No$, a function $\Psi_{mn}\in \CM[0,\infty)$, associated to a  measure $\nu_{nm}$  with finite total mass $\nu_{nm}\big([0,\infty))=\Psi(1/m)$, such that for every $l\in \NN$,
\begin{eqnarray}
\tau_{\frac{1}{m}}\Psi(\frac{l}{n})&=&\Psi(\frac{n+ml}{mn})=\Psi_{mn}(n+ml)=
\int_{[0,\infty)}
e^{-(n+ml)\,u}\, \nu_{nm}(\du)\nonumber\\
&=& \int_{[0,\infty)} e^{-\frac{n+ml}{mn}\,v}\, \widetilde{\nu}_{n,m}(\dv)
=\int_{[0,\infty)} e^{-lv}\,\overline{\nu}_{n,m}(\dv)\label{uv}
\end{eqnarray}
where $\widetilde{\nu}_{n,m}$ is the image of $\nu_{nm}$ by the change of variable $u=\frac{v}{n}$.  Taking $l=0$ in  (\ref{uv}), it is immediate that the measure $$\overline{\nu}_{n,m} (\dv):=e^{-\frac{v}{m}} \, \widetilde{\nu}_{n,m}(\dv)$$ is also a measure  with finite total mass $\overline{\nu}_{n,m}\big([0,\infty)\big)=\Psi(\frac{1}{m})$.

It is now evident, by proposition \ref{proposition4}, that for every $m$, the function $\tau_{\frac{1}{m}}\Psi$ is completely monotone on
$[0,\infty)$ for every $m\in \No$. Using Remark \ref{clos} (ii), we conclude that $\Psi \in \CM(0,\infty)$.
\end{proof}
\begin{proof}[Proof of Theorem~\ref{theorem5}] We tackle the proof with the necessity part:  the holomorphy condition (i) is in Proposition \ref{holom}
and the second condition stems from Theorem \ref{WA}. Proof of the sufficiency part  is based on Blaschke's result stated in Corollary  \ref{blas}, used with some care, because Bernstein function are not bounded in general. By Proposition \ref{proposition3}, it is enough to check whether the function
$$\lambda \mapsto \theta\Phi(\lambda):=\Delta_1\Phi(0) -\Delta_1\Phi(\lambda)=\Phi(1)-\Phi(0)+\Phi(\lambda)-\Phi(\lambda+1)$$
belongs to $\BF_b^0$ in order to show that $\Phi \in \BF$.  We argue as follows:

1-   representation (\ref{refor}) gives
$$\Phi(k)= q + dk+ \int_{(0,\infty)} (1-e^{-kx})\mu(\dx),\quad  k\in \NN,$$
and allows to define the function
\begin{equation*}\label{pio}
\overline{\Phi}(\lambda)= q + d\lambda+ \int_{(0,\infty)} (1-e^{-\lambda x})\mu(\dx),\quad  \lambda\in [0,\infty),
\end{equation*}
and then, by Proposition \ref{proposition3}, $\,\theta\overline{\Phi} \in \BF_b^0$;

2- the sequences ${\big(\theta\Phi(k)\big)}_{k\geq 0}$ and ${\big(\theta\overline{\Phi}(k)\big)}_{k\geq 0}$ are equal;

3- boundedness condition in (a)  yields boundedness of the function the extension of  $\theta\Phi $, boundedness of the function the extension of  $\theta\overline{\Phi}$ stems from Proposition \ref{holom};

4- Corollary \ref{blas} insures  that the extensions of the functions $\theta\Phi$ and $\theta\overline{\Phi}$ are equal on $(0,\infty)$ and also on since $\theta\Phi(0)=\theta\overline{\Phi}(0)=0$. Then, $\theta\Phi\in \BF_b^0$.
\end{proof}
\begin{proof}[Alternative proof of Theorem~\ref{theorem5}] As in the alternative proof of  Theorem~\ref{theorem4}, Gregory-Newton expansion approach  works. Do as in the proof Theorem~\ref{theorem5} until point 3- and use Proposition \ref{holom} to conclude in a  point 4- that both extensions of  $\theta\Phi$ and $\theta\overline{\Phi}$ share the Gregory-Newton expansion and then are equal.
\end{proof}
\begin{proof}[Proof of Proposition~\ref{proposition5}] The necessity part is comes from (\ref{bc}). The sufficiency part is an adaptation of
the proof of Proposition \ref{proposition4}. From (\ref{uniform}), we have
$$1-e_n(\lambda, u)= 1-e^{-\alpha_n[\frac{\lambda}{\alpha_n}]\,u}\longrightarrow 1- e^{-\lambda\,u} \;\,\mbox{uniformly in $u\geq 0$ when $n\to \infty$.}$$
Notice that
$$\Phi\Big(\alpha_n[\frac{\lambda}{\alpha_n}]\Big) =  \Phi_n\Big(\alpha_n[\frac{\lambda}{\alpha_n}]\Big)= \int_{[0,\infty)} \big(1-e^{-\left[\frac{\lambda}{\alpha_n}\right] u})\big)\, \mu_n(\du),$$
where $\mu_n$ is the representative  measure of $\Psi_n$. By  the change  variable $u=\alpha_n v$,  the representation
$$\Phi\Big(\alpha_n[\frac{\lambda}{\alpha_n}]\Big) =\int_{[0,\infty)} \big(1-e_n(\lambda, u)\big)\, \widetilde{\mu}_n(\du)$$
holds true where $\widetilde{\mu}_n$  being  a finite measure  with total mass $\widetilde{\mu}_n\big((0,\infty)\big)=\lim_{\lambda \to \infty}\Phi(\lambda)<\infty$ due to the monotone convergence theorem applied along $\lambda \to \infty$.  The rest of the proof is continued exactly as in proof of Proposition \ref{proposition4} through the limit $\Phi(\lambda)= \lim_{n\to \infty} \Phi\left(\alpha_n[\lambda/\alpha_n]\right).$
\end{proof}
\begin{proof}[Proof of Corollary~\ref{corollary5}] The implication $(a) \Longrightarrow (b)$ is justified by (\ref{bc}) and $(b) \Longrightarrow (c)$
being immediate, we just need to prove $(c) \Longrightarrow (a)$. By In order to show that $\Phi \in \BF$, it is enough, by Proposition
\ref{proposition3}, to check that for every fixed $m\in\No$, the function
$$\theta_{\frac{1}{m}}\Phi(\lambda)=\Phi(\frac{1}{m})-\Phi(0)+\Phi(\lambda)-\Phi(\frac{1}{m}+\lambda), \;\lambda \geq 0,$$
belongs to $\BF_b^0$. By assumption there exists  for each $n\in \No$,  a function $\Phi_{mn}\in \BF$, having triplet of characteristics $(q_{mn}, d_{mn},\mu_{mn})$, such that the following representation holds true for all $k\in \NN$:
\begin{eqnarray}
\Phi(\frac{1}{m}+\frac{k}{n})- \Phi(\frac{k}{n})&=& \Phi(\frac{mk+n}{mn})- \Phi(\frac{mk}{mn}) =  \Phi_{mn} (mk+n ) -\Phi_{mn}(mk)\nonumber\\
&=& d_{mn} \,n + \int_{(0,\infty)} e^{-mk u}\,(1-e^{-n
u})\,\mu_{mn}(\du).\label{0}
\end{eqnarray}
Representation (\ref{0}) shows that the sequence $k\mapsto \Phi(\frac{1}{m}+\frac{k}{n})- \Phi(\frac{k}{n})$ is positive and decreasing then is converging. Similarly, we have
\begin{eqnarray}
\theta_{\frac{1}{m}}\Phi(\frac{k}{n})&=&\Phi(\frac{1}{m})-\Phi(0)+\Phi(\frac{k}{n})-\Phi(\frac{1}{m}+\frac{k}{n}) \label{1}\\
&=&\Phi(\frac{n}{mn})-\Phi(0)+\Phi(\frac{mk}{mn})-\Phi(\frac{mk+n}{mn})\nonumber\\
&=&\Phi_{mn}(n)-\Phi_{mn} (0)+\Phi_{mn}(mk)- \Phi_{mn} (km+n )\nonumber\\
&=&\int_{(0,\infty)}(1-e^{-km u})\,(1-e^{-n
u})\,\mu_{mn}(\du).\nonumber
\end{eqnarray}
Making the change of variable $u=v/m$ in (\ref{4}), we retrieve with the image $\widetilde{\mu}_{mn}$ of $\mu_{mn}$ that
\begin{equation}\theta_{\frac{1}{m}}\Phi(\frac{k}{n})=\int_{(0,\infty)}(1-e^{-k u})\,(1-e^{-\frac{n}{m}u})\,\widetilde{\mu}_{mn}(\du),\quad \forall k\in \NN. \label{4}\end{equation}
Representation (\ref{1}) and continuity of $\theta_{\frac{1}{m}}\Phi$ justifies that $\lim_{x\to \infty }\theta_{\frac{1}{m}}\Phi(x) = \lim_{k\to \infty } \theta_{\frac{1}{m}}\Phi(\frac{k}{n})$ is finite. Then, the monotone convergence theorem insures applied in (\ref{4}) gives that
$$\int_{(0,\infty)} (1-e^{-\frac{n}{m}u})\,\widetilde{\mu}_{mn}(\du)= \lim_{x\to \infty }\theta_{\frac{1}{m}}\Phi(x).$$
i.e. the measure $\,(1-e^{-\frac{ u}{m}})\,\widetilde{\mu}_{mn}(\du)$ is finite with total mass $\lim_{x\to \infty }\theta_{\frac{1}{m}}\Phi(x)$. We conclude that $\theta_{\frac{1}{m}}\Phi$ satisfies the condition of Proposition \ref{proposition5} and then belongs to $\BF_b^0$.
\end{proof}
\section{Bernstein Self-decomposability property of  functions is also recognized by their restriction on $\NN$} \label{self}
During the redaction of this paper,  we felt it important to clarify the probabilistic notion of  infinite divisibility and self-decomposability of non0-negative random variables. The probabilistic point of view is well presented in the book Steutel and van Harn in \cite{steutel}. Every Bernstein function $\Phi$, null in zero, is the cumulant function (i.e. Laplace exponent) of an {\it infinitely divisible} non-negative random variable $Z$, i.e.
$$\er[e^{-\lambda Z}]:=\int_{[0,\infty)} e^{- \lambda x}\,\mathbb{P}(Z\in \dx)= e^{-\Phi(\lambda)},\quad \lambda \geq 0. $$
The latter is equivalent to the existence,  for every integer $n$, of non-negative i.i.d random variables  $Z_1^n, \cdots, Z_n^n$ such that
$Z \simdis Z_1^n + \cdots +Z_n^n,$ or also to the fact that  the function
$$\lambda \mapsto \left(\er[e^{-\lambda Z}]\right)^t \quad \mbox{is completely monotone for every $t>0$}.$$

In \cite{steutel}, Steutel and van Harn present class of non-negative self-decomposable r.v.'s by those random variables $X$, such that
for every $c \in (0,1)$, the function
\begin{equation}\label{mp}
\lambda \mapsto \Psi_c(\lambda)=\er[e^{-\lambda X}]/\er[e^{-c\lambda
X}],
 \end{equation}
belongs to $\CM[0,\infty)$. The latter is equivalent to the  existence,  for each $c\in (0,1)$, of a r.v. $Y_c$  independent from $X$ such that the folloowing identity in law  holds true
\begin{equation*}\label{xcy}
X\simdis c\, X +  Y_c
\end{equation*}
Necessarily the r.v. $X$ is infinitely divisible and  is called a self-decomposable r.v. Its  cumulant function  $\Phi(\lambda)=-\log \er[e^{-\lambda X}],\;\lambda\geq 0$ (necessarily differentiable) satisfies  (\ref{mp}) or  equivalently it satisfies  $3)(b)$ in Proposition \ref{proposition6} below, for this reason,   $\Phi$ is called a self-decomposable Bernstein function.  Another characterization of $\Phi$ is a specification  of the form (\ref{berep}) with $q=0$ and the L\'evy measure of the form $\nu({\rm d}x)=x^{-1}k(x){\rm d}x, \, x>0$ with $k$ a decreasing function (see \cite{sato} for more account).

We denote $\CF$  the class of cumulant functions of probability measures, i.e.:
$$\CF := \{\lambda \mapsto \phi(\lambda)= -\log \er[e^{-\lambda Z}]= -\log\int_{[0,\infty)} e^{-\lambda x} \, \pr(Z\in \dx),\  \;
Z \;\mbox{\it a non-negative r.v.}\}\,.$$
\begin{remark} It is clear that

(i)  $\CF$ is stable by addition (it stems from the addition of independent random variables), is closed under pointwise limits (this is the convergence in distribution) and also stable by the operators $\sigma_c$ and $\tau_c$ introduced in Section \ref{basic}.

(ii) $\Phi \in \BF$ if and only if $t \big(\Phi -\Phi(0)\big)\in \CF\,$ for every $t>0$. $\phi \in \CF$ if and only if $1-e^{-\phi} \in \BF_b$. The latter yields  $\Phi \in \CF$ if and only if $1-e^{-t\Phi} \in \BF_b\,$ for every $t>0$.

(iii) Observe that $\Phi \in \BF$ if and only if  $(1-e^{-\epsilon_n \Phi})/\epsilon_n  \in \BF_b$ for some positive sequence $\epsilon_n$ tending to zero.  To see the claim, use closure property under pointwise limits of $\BF$ (Corollary 3.9 p. 29 in \cite{SSV}) together with   $\Phi=\lim_{n\to \infty} \big(1-e^{- \epsilon_n\Phi}\big)/ \epsilon_n$.  One can deduce that $\Phi$ belongs to $\BF$ if and only if $\epsilon_n \Phi$ belongs to $\CF$
for some positive sequence $\epsilon_n$ tending to zero.
\label{cf}\end{remark}

We have the following useful  result related to (\ref{bc}):
\begin{proposition} Let $\Phi: [0,\infty)\longrightarrow \RP$ and $\rho_c\Phi(\lambda):=(\sigma-\sigma_c)\Phi(\lambda)= \Phi(\lambda)-\Phi(c\lambda)$, $c\in (0,1)$.

\noindent 1) If $\Phi$  is continuous at the neighborhood of $0$ and $\rho_c\Phi  \in \CF$ (respectively $\BF$) for some $c\in(0,1)$, then $\Phi$ belongs to $\CF$ (respectively $\BF$).

\noindent 2) Assume $\Phi$ is continuous at the neighborhood of $0$, then the following assertions are equivalent:

(a) $\rho_c\Phi \in \BF$ for every $c\in(0,1)$;

(b) $\rho_c\Phi \in \CF$ for every $c\in(0,1)$;

(c)  $\Phi$ is differentiable on $(0,\infty)$ and $\lambda \mapsto \lambda \Phi'(\lambda) \in\BF$.
\label{proposition6}\end{proposition}
\begin{proof}[Proof of Proposition~\ref{proposition6}] 1) If  $\rho_c\Phi$ belongs to $\CF$ (respectively $\BF$),  then for every $n\in \NN$,
$$\lambda \mapsto\rho_{c^n}\Phi(\lambda)=\Phi(\lambda)-\Phi(c^n\lambda) = \sum_{i=0}^{n-1} \Phi(c^k\lambda)-\Phi(c^{k+1}\lambda)= \sum_{i=0}^{n-1} \rho_{c}\Phi(c^k\lambda)$$
belongs to $\CF$ (respectively $\BF$). By closure of $\CF$ (respectively $\BF$) and using the fact that $\Phi$ is continuous at $0$, deduce  that $\Phi-\Phi(0) =\lim_{n\to \infty}\rho_{c^n}\Phi \in \CF$ (respectively $\BF$).

2) $(a)\Longrightarrow (b)$: By Remark \ref{cf} (ii), $\rho_c\Phi \in \BF$ and is null at zero,  then  $\rho_c\Phi \in \CF$.

\noindent $(b)\Longrightarrow (c)$: Since $\rho_c\Phi \in \CF$ for all $c\in (0,1)$, then by 1),  $\Phi \in \CF$ and then differentiable. Further, by Remark \ref{cf} (ii),  $\rho_c\Phi \in \CF$ for all $c\in (0,1)$ implies  to $\big(1-e^{-\rho_c\Phi}\big)/(1-c)\in\BF$ for all $c\in (0,1)$. Letting $c\to 1-$, we get, by closure of $\BF$ again, that the  $\lambda\mapsto \lambda \Phi'(\lambda) = \lim_{c\to 1-} \big(1-e^{-\rho_c\Phi(\lambda)}\big)/(1-c) \in\BF$.

\noindent  $(c)\Longrightarrow (a)$: The function  $x\mapsto \Phi_0(x) =x \Phi'(x)\in\BF$.  Write $\lambda\mapsto \rho_c\Phi(\lambda)= \int_c^1 \Phi_0(\lambda x)\,\frac{\dx}{x}$ for every $c\in (0,1)$, observe that differentiability under the integral is well justified and  the alternating property of the function under the last integral allows to conclude that $\rho_c\Phi\in \BF$.
\end{proof}

We are then able to state a Corollary to Theorem \ref{theorem5} and Proposition \ref{proposition6}:
\begin{corollary} Let function $\Phi:[0,\infty)\to [0,\infty)$ admitting a finite limit at $0$. Then

1) $\Phi$ is a Bernstein function if and only if it admits holomorphic extension on the half plane $Re(z)>0$ and ${\left(\Phi(k)-\Phi(ck)\right)}_{k\geq 0}$ is completely alternating  and minimal for some  $c\in (0,1)$.

2) $\Phi$ is a self-decomposable Bernstein function if and only if it admits holomorphic extension on the half plane $Re(z)>0$ and one
the following holds

(a) the sequence ${\left(\Phi(k)-\Phi(ck)\right)}_{k\geq 0}$ is completely alternating  and minimal for all  $c\in (0,1)$;

(b)  the sequence ${\left(k\Phi'(k)\right)}_{k\geq 0}$ is completely alternating and minimal.
\label{corollary7}\end{corollary}
\begin{remark}
The main contribution in \cite{mai} consists in Theorem 1.1 where it was stated in case  $\Phi(0)=0$: $\Phi$ is a self-decomposable
Bernstein function if and only if
$${\left(\Phi(xk)-\Phi(yk)\right)}_{k\geq 0}\; \;\mbox{is completely alternating for every}\; \; x>y>0.$$
No minimality nor holomorphy conditions were required in \cite{mai}. In our work, these conditions appeared to be the lowest price to pay in order
to fix $x=1$ or to have the non parametric  characterization  2)(b) which clarifies their discussion at the end of section 1 in \cite{mai}.
\end{remark}

\end{document}